\documentclass[journal]{IEEEtran}
\usepackage{array,url}
\usepackage[margin=1.2in]{geometry}
\usepackage[geometry]{ifsym}
\usepackage{amssymb, amsfonts, amsmath, amsthm, amscd, mathrsfs, mathtools, color}
\usepackage{todonotes, comment, verbatim, fullpage}
\usepackage{graphicx, float, epsfig, enumerate, fixmath, bbm, latexsym, cite, scrextend}
\usepackage{pgfplots}
\pgfplotsset{compat=1.12}
\usepackage{subcaption}

\usepackage[hidelinks]{hyperref}
\usepackage{xcolor}
\hypersetup{
    colorlinks,
    linkcolor={blue!100!black},
    citecolor={green!50!black},
    urlcolor={red!80!black}
}

\def \eu{\mathrm{e}}

\newcommand{\bbC}{\mathbb{C}}
\newcommand{\rmd}{\mathrm{d}}
\newcommand{\bbE}{\mathbb{E}}\newcommand{\rme}{\mathrm{e}}

\newcommand{\bbN}{\mathbb{N}}

\newcommand{\bbR}{\mathbb{R}}

\newcommand{\sfA}{\mathsf{A}}

\newcommand{\sfK}{\mathsf{K}}

\newcommand{\sfm}{\mathsf{m}}

\newcommand{\sign}{\mathsf{sign}}

\newcommand{\erf}{\mathrm{erf}}

\newcommand{\E}{\mathbb{E}}

\newtheorem{thm}{Theorem}

\newtheorem{lem}[thm]{Lemma}
\newtheorem{prop}[thm]{Proposition}
\newtheorem{rem}{Remark}

\begin{document}

\title{$L^1$ Estimation: On the Optimality of Linear Estimators} 

\author{Leighton P. Barnes, Alex Dytso, Jingbo Liu, and H. Vincent Poor
\thanks{
Leighton P. Barnes is with the  Center for Communications Research, Princeton, NJ 08540, USA (e-mail: l.barnes@idaccr.org). Alex Dytso is with Qualcomm Flarion Technology, Inc.,
Bridgewater, NJ 08807, USA (e-mail: odytso2@gmail.com).
Jingbo Liu is with the Department of Statistics and the Department of Electrical and Computer Engineering, University of Illinois, Urbana-Champaign, IL 61820, USA (e-mail: jingbol@illinois.edu).
H. Vincent Poor is with the Department of Electrical and Computer Engineering, Princeton University, 
                   Princeton, NJ 08544,  USA (e-mail: poor@princeton.edu).
}
\thanks{Part of this work was presented at the  2023 IEEE International Symposium on Information Theory (ISIT)~\cite{BDP_ISIT2023}.
}
\thanks{This work was supported in part by the Grants CNS-2128448 and ECCS-2335876.}
}

\maketitle

\begin{abstract}
Consider the problem of estimating a random variable $X$ from noisy  observations $Y = X+ Z$, where $Z$ is standard normal, under the $L^1$ fidelity criterion. It is well known that the optimal Bayesian estimator in this setting is the conditional median. This work shows that the only prior distribution on $X$ that induces linearity in the conditional median is Gaussian.

Along the way, several other results are presented. In particular, it is demonstrated that if the conditional distribution $P_{X|Y=y}$ is symmetric for all $y$, then $X$ must follow a Gaussian distribution. Additionally, we  consider  other $L^p$ losses  and observe the following phenomenon: for $p \in [1,2]$, Gaussian is the only prior distribution that induces a linear optimal Bayesian estimator, and for $p \in (2,\infty)$,  infinitely many prior distributions on $X$ can induce linearity.  Finally, extensions are provided to encompass noise models leading to conditional distributions from certain exponential families.
 \end{abstract}

\section{Introduction}
The theory of linear estimation plays a central role in Bayesian estimation. Linear estimators are easy to deploy and are thus attractive from a practical point of view.  From a theoretical point of view, linear estimators often serve as useful benchmarks for understanding the performance of other more complex estimators.  Furthermore, they appear as key objects in the study of Bayesian topics such as exponential families and conjugate priors. Thus characterizing the optimality of linear estimators is important from both practical and theoretical points of view. 

In the Bayesian setting, whether a given estimator is optimal or not depends highly on the chosen fidelity criterion.  To make things concrete, suppose we seek to estimate a scalar random variable $X \in \bbR$ from a noisy observation $Y\in\bbR$.  For the time being, we will focus on the simple but already rich noise model  
 \begin{equation}
Y=X+Z, \label{eq:Gaussian_noise_model} 
\end{equation} 
where $Z$ is \emph{standard normal} independent of $X$. Later on, we will also consider more general noise models. The fidelity criterion determines the nature of the optimal estimator.  In particular, for $L^2$ and $L^1$ error measures, it is well-known that the optimal estimators are given by the \emph{conditional mean} and the \emph{conditional median}, respectively,  that is
\begin{align}
 \bbE[X|Y] &= \arg \min_{ f : \, \E[ |f(Y)| ^2]<\infty}  \E \left[  \left| X - f(Y) \right|^2 \right], \label{eq:L2_optimizer} \\
 \sfm(X|Y)  &= \arg \min_{ f : \, \E[ |f(Y)|]<\infty}  \E \left[  \left| X - f(Y) \right| \right]. \label{eq:L1_optimizer} 
 \end{align} 
 The conditional mean and median of $X$ given $Y=y$ are defined as 
     \begin{align}
     \bbE[X|Y=y]&= \int x  \,  F_{X|Y=y}(\rmd x), \, y\in \mathbb{R}, \\
     \sfm(X|Y=y) &= F^{-1}_{X|Y=y} \left( \frac{1}{2} \right), \, y \in \mathbb{R}, \label{eq:definition_of_median}
     \end{align} 
 where $F_{X|Y=y}$ is the conditional cumulative distribution function (cdf) and  $F^{-1}_{X|Y=y}$ is the conditional quantile function\footnote{ Recall that for a random variable $U$ the \emph{quantile function} or the \emph{inverse cumulative distribution function} (cdf) is defined as $
F_U^{-1}(p)= \inf \{x \in \mathbb{R}: p \le F_U(x) \}, \, p \in (0,1).$} of $X$ given $Y=y$.  Therefore, under the $L^2$ criterion, the optimality of linear estimators reduces to characterizing whether or not there exists a constant $a$ and a prior distribution for $X$ such that for (almost) all $y \in \bbR$, 
\begin{equation}
\bbE[X|Y=y]= a y.  
\end{equation} 
Similarly, for the $L^1$ criterion, the optimality of linear estimators reduces to characterizing  whether there exists a constant $a$ and a prior distribution for $X$ such that for (almost) all $y \in \bbR$,
\begin{equation}
 \sfm(X|Y=y) = ay.  
\end{equation}

For the case of $L^2$ error, the problem of identifying the set of distribution on $X$ that would induce a linear conditional mean has been well understood for several decades.  In fact, for the Gaussian noise setting, in Appendix~\ref{app:proofs_of_linearity} we provide four different ways of showing that the only prior that induces linearity is the Gaussian with zero mean and variance $\frac{a}{1-a}$, i.e., $X \sim \mathcal{N}(0,\frac{a}{1-a})$, and that the only admissible $a$ values lie in the interval $[0,1]$.  The problem is also well understood beyond the Gaussian noise case. For example, when $P_{Y|X}$ belongs to an \emph{exponential family}, it is known that the conditional mean $\bbE[\psi'(X)|Y]$ where $\psi$ is the log-partition function is linear if and only if $X$ is distributed according to a \emph{conjugate prior} \cite{diaconis1979conjugate,chou2001characterization}.  Conjugate priors are used to model a variety of phenomena in statistical and machine-learning applications \cite{bishop2006pattern}.  

For additive noise channels, i.e., $Y=X+N$, where $N$ is not necessarily Gaussian, the authors of \cite{akyol2012conditions} characterized necessary and sufficient conditions for the linearity of the optimal Bayesian estimators for the case of $L^p$ Bayesian risks (i.e., $\bbE[ |X-f(Y) |^p]$) with $p$ taking only \emph{even} values. More specifically, the authors of   \cite{akyol2012conditions} found the characteristic function of $X$ as a function of the characteristic function of $N$.  These results, however, do not generalize to our case of  $p=1$. 

Finally, for the $L^2$ case, in addition to uniqueness results, we also have \emph{stability} results.  In particular,    for the Gaussian and Poisson noise models,  if the conditional expectation is close to a linear function in the $L^2$ distance, then the distribution of $X$ needs to be close to a matching prior (Gaussian for Gaussian noise and gamma for Poisson noise) in the \emph{L\'evy distance}  \cite{du2018strong,dytso2020estimation}.

Interestingly,  for the $L^1$ case, there appears to be no full answer in the existing literature. While $X \sim \mathcal{N}(0,\frac{a}{1-a})$ clearly induces a linear conditional median, to the best of our knowledge, there are no previous results that guarantee that this is the only prior that induces linearity. 
The aim of this work is to close this gap and show that the Gaussian distribution is the only one inducing linearity of the conditional median.   Moreover, we will provide several equivalent perspectives on this problem related to integral operator theory and convolution equations. Near the end, we will also examine other $L^p$ losses and will show that for $p \in[ 1,2]$, Gaussian is the only distribution that induces linearity of the optimal Bayesian estimator, and for $p \in (2,\infty)$,  multiple prior distributions on $X$ can induce linearity. The point $p=2$ is an interesting phase transition point which appears to not have been noted in prior literature. In handling $L^p$ losses for all $p\in [1,\infty)$, instead of just even values, we work in a more general setting than \cite{akyol2012conditions}. In that work, they show that for even $p$, there are no priors \emph{with the same variance as the noise}, other than the Gaussian, that induce linearity of the optimal Bayesian estimator. Since we do not make the equal variance assumption, our results are not contradictory and give a more complete understanding of the scenario.

The conditional median plays an important role in our analysis, and for a detailed study of the properties of the conditional median in the   abstract  measure theoretic setting, the interested reader is referred to \cite{tomkins1975conditional,ghosh2006probabilistic}. For recent applications of the conditional median, the interested reader is referred to \cite{medarametla2021distribution} and references therein.  In this work we will also rely on a number of Fourier and complex analysis techniques. Similar techniques have been successful in proving a number of Gaussian-characterizing properties, such as the Bernstein theorem \cite{bern} and the L\'evy-Cram\'er theorem \cite{cramer1936eigenschaft},
which, in turn, show that Gaussian distributions arise in a number of optimization problems in information theory \cite{geng2014capacity,courtade2014extremal,liu2017information, liu2018forward, anantharam2022unifying, liu2023stability, mahvari2023stability}.

%
%
%



\subsection{Outline and Contributions } 

The paper outline and contributions are as follows. In Section~\ref{sec:Preliminary_Observations}, we start by making some preliminary observations about characterizing which prior distributions lead to linear conditional medians. In Section~\ref{sec:Gaussian_is_a solution}, Proposition~\ref{prop:Gausssian_is_a_solution}, we show that a Gaussian prior distribution does indeed yield a linear conditional median. Section~\ref{sec:Equivalent_Condition}, Proposition~\ref{prop:Equivalent_convolution_problem} provides an equivalent condition to linearity in terms of a convolution; and Section~\ref{sec:operator_persepective} discusses finding the nullspace of the corresponding linear operator. In Section~\ref{sec:Main_Results}, we present our main results and show that the Gaussian distribution is the only prior distribution that induces a linear conditional median.  The proof of the main theorem uses a Fourier approach to solve the convolution equation, with care being needed to establish growth estimates so that the Fourier transform of the relevant measures can be made sense of using the machinery of tempered distributions. The non-negativity of the measure plays a key role in establishing these growth estimates, and without it there are indeed counterexamples. Finally, in Section~\ref{sec:Discussion} we conclude by discussing the Poisson noise case,  other $L^p$ losses, exponential families, and possible multidimensional extensions.

\section{The Problem Setup and Some Preliminary Observations} 
\label{sec:Preliminary_Observations}

In this section, we begin by providing preliminary observations about the problem, derive a necessary and sufficient condition for linearity of the conditional median to hold, and try to point out the reasons why solving the problem is a challenging task.  Along the way, we also derive some results that might be of independent interest.  

    \subsection{The Problem Setup}
     Despite considerable research into linear estimators, to the best of our knowledge, the question of identifying the set of prior distributions on $X$ that ensure that $ \sfm(X|Y)$  is a linear function of $Y$ has not been characterized.   In this work, we seek to close this gap.   Formally, we seek to answer the following question: \emph{For a given $a \in \bbR$, what is the set of distributions on the input $X$ that ensure  that for all $y \in \mathbb{R}$} 
    \begin{equation}
    \sfm(X|Y=y) = ay \; ?  \label{eq:linearity_condition} 
    \end{equation} 
   Unless otherwise noted, we will focus on the Gaussian  in \eqref{eq:Gaussian_noise_model}.  However, interestingly, we will also be able to adapt the Gaussian methods to imply similar results for a subset of exponential families.

It is well known that for the conditional expectation, under the model in \eqref{eq:Gaussian_noise_model}, 
\begin{equation}
\bbE[X|Y=y]=ay, \, \forall y\in \bbR,
\end{equation}  
if and only if $ a \in [0,1)$ and $X \sim \mathcal{N}(0,\frac{a}{1-a})$. The authors of this paper are aware of five district ways of showing this fact; four of these methods, some of which are new,  are provided in Appendix~\ref{app:proofs_of_linearity}.  However, none of these techniques appear to be generalizable to the conditional median setting. Indeed, this work develops a new technique to establish the linearity of the conditional median.

\subsection{On the Admissible Values of $a$} 
\label{sec:Admissalbe_a}

We next show that the admissible values of $a$ that satisfy \eqref{eq:linearity_condition} must be in $[0,1]$. This is done for all $L^p$ losses with $p \ge 1$.

\begin{thm}\label{thm:admissable_a} Let $p \ge 1$. Then, 
\begin{equation}
 \min_{a \in \bbR} \bbE[ | X-aY|^p ]= \min_{{a \in [0,1] }} \bbE[ | X-aY|^p ].
\end{equation} 
In other words, the \emph{admissible} values of $a$ lie in $[0,1]$. 
\end{thm} 
\begin{proof} 
We will assume, without loss of generality, that $\mathbb{E}[X]=0$. This can be done by constructing a version of the measure $P_X$ that is symmetric about the origin by averaging it with its time-reverse. This measure will have mean-zero, and will give the same values for $\bbE[ | X-aY|^p ]$. {We will furthermore assume these expectations are finite: if $\bbE[ | X-aY|^p ]$ is infinite for all $a\notin [0,1]$ then the theorem is trivial, and if there is an $a_0\notin [0,1]$ for which it is finite then it must be finite for all $a$. In particular, if we take $a=0$ then we can bound the $p-1$ and $p$ moments of $X$ as follows:
\begin{equation}\label{eq:moment_bound2}
\bbE [|X|^{p-1}] \leq 1 + \bbE [|X|^p] < \infty \; .
\end{equation}

Let 
\begin{equation}
f(a)=\bbE[ | X-aY|^p ],
\end{equation}
which is differentiable for $a\neq 0$ (and continuous for all $a$) in view of the fact that $X-aY= (a-1)X +aZ$ is a continuous random variable with a density, and the set of points for which the function $a \mapsto | X-aY|^p $ is not differentiable has measure zero.
Then, by letting
\begin{equation}
m(t) = p \, \sign(t) | t|^{p-1}, \, p \ge 1, \, t \in \bbR, 
\end{equation}
we have
\begin{align}
f'(a)&=-  \bbE[ m(X-aY) Y] \\
& = \bbE[ m(aY-X) Y] \\
&=  \bbE \left[ m\left( (a-1)X +aZ  \right) (X+Z) \right].
\end{align}
The interchange between limiting operations can be made rigorous by using \eqref{eq:moment_bound2} to show a dominating measurable function.

We will show that the function $f(a)$ is \emph{non-decreasing} for $a\ge1$ and \emph{non-increasing} for $a \leq 0$. Thus, we will reduce our search space to $a \in [0,1]$. To aid our proof, recall the FKG inequality  (see for example \cite{kemperman1977fkg}): for two independent random variables $U$ and $V$ and two coordinate-wise non-decreasing functions $f_1$ and $f_2$, we have that
\begin{equation}
\bbE[ f_1(U,V) f_2(U,V) ] \ge  \bbE[ f_1(U,V)  ] \bbE[f_2(U,V) ] ,
\end{equation} 
or equivalently, if  $f_1$ and $f_2$ are coordinate-wise non-increasing functions, then
\begin{equation}
\bbE[ f_1(U,V) f_2(U,V) ] \le  \bbE[ f_1(U,V)  ] \bbE[f_2(U,V) ] .
\end{equation}

Now, assume that $a \ge 1$; then
\begin{align}
f'(a)&=\bbE \left[ m \left( (a-1)X +aZ  \right) (X+Z) \right] \\
& \ge    \bbE \left[   m \left( (a-1)X +aZ  \right) \right]  \bbE \left[ 
 X+Z  \right] \label{eq:FKG_fist_time} \\
&=0, \label{eq:Using_zero_mean_assumption_first}
\end{align} 
 where in \eqref{eq:FKG_fist_time} we have used the FKG inequality together with the fact that $(X, Z) \mapsto m \left( (a-1)X +aZ  \right)$ and $(X,Z) \mapsto (X+Z)$  are coordinate-wise non-decreasing; and in \eqref{eq:Using_zero_mean_assumption_first} we have used that $\bbE[X]= \bbE[Z]=0$.  

Now, assume that $a\le 0$; then 
\begin{align}
f'(a)&=\bbE \left[ m \left( (a-1)X +aZ  \right) (X+Z) \right] \\
&\le     \bbE \left[   m \left( (a-1)X +aZ  \right)     \right]  \bbE \left[   X+Z    \right]  \label{eq:FKG_third_use}\\
&=0 ,  \label{eq:Using_zero_mean_assumption_second}
\end{align}
 where in \eqref{eq:FKG_third_use} we have used the FKG inequality together with the fact that $(X, Z) \mapsto m \left( (a-1)X +aZ  \right)$ and $(X,Z) \mapsto (X+Z)$  are coordinate-wise non-decreasing; and in \eqref{eq:Using_zero_mean_assumption_second} we have used that $\bbE[X]= \bbE[Z]=0$.  } 

Thus, we can assume that $a \in  [0,1]$. 
\end{proof}

    \subsection{Gaussian $X$ is a Solution}
    \label{sec:Gaussian_is_a solution}
    We begin by showing that the set of distributions that satisfies \eqref{eq:linearity_condition} is not empty. 
    \begin{prop}\label{prop:Gausssian_is_a_solution} If $0\le a <1$, a Gaussian random variable $X \sim \mathcal{N}(0,\sigma_X^2)$ satisfies  \eqref{eq:linearity_condition} if $\sigma_X^2= \frac{a}{1-a}$. 
    \end{prop}
    \begin{proof} Suppose that  $X \sim \mathcal{N}(0,\sigma_X^2)$; then the conditional distribution $X|Y=y \sim \mathcal{N}( \frac{\sigma^2_X}{1+ \sigma_X^2 } y,  \frac{\sigma^2_X}{1+ \sigma_X^2 }) $.   Since Gaussian distributions are symmetric, the conditional median and conditional mean coincide, and we have that
    \begin{equation}
    \sfm(X|Y=y)= \E[X|Y=y]=  \frac{\sigma^2_X}{1+ \sigma_X^2 } y.
    \end{equation}
    Solving for $\sigma_X^2$ concludes the proof. 
    \end{proof}
    
   In Proposition~\ref{prop:Gausssian_is_a_solution}, for the case of $a=0$, and for the rest of the paper, we do not distinguish between point measures and   Gaussian measures with zero variance and treat them as the same objects. 
   
The proof of Proposition~\ref{prop:Gausssian_is_a_solution} relied on the fact that if $X$ is Gaussian, then $X|Y=y$ is a symmetric distribution\footnote{The random variable $U$ is said to have symmetric distribution if there exists a constant $c$ such that $U+c\stackrel{d}{=} - (U+c)$ where $\stackrel{d}{=}$  denotes equality in distribution.   } for all $y$  and, hence, the mean and the median coincide. The next result, which might be of independent interest, shows that this construction works only in the Gaussian case.

\begin{thm}\label{thm:symmetry_of_post} If $X$ is Gaussian, then $X|Y=y$ is symmetric for all $y$.  Conversely, if $X|Y=y$ is symmetric for all $y \in S$ where $S$ is a subset of $\bbR$ that has an accumulation, then $X$ is Gaussian. 
\end{thm} 

\begin{proof}
See Appendix~\ref{app:thm:symmetry_of_post}. 
\end{proof}

\subsection{An Equivalent Condition via Convolution}
\label{sec:Equivalent_Condition}
In this subsection, we derive a condition that is equivalent to \eqref{eq:linearity_condition}.  Our starting place is the following condition akin to the orthogonality principle  \cite{dytso2017minimum,akyol2012conditions}:  a function $f$  is a median if and only if 
\begin{equation}
\bbE \left[ \sign \left(X - f(Y)  \right) \eta(Y) \right]=0, \label{eq:Orthogonality_like_property}
\end{equation} 
for all $\eta$ such that  $ \bbE[ |\eta(Y)|]<\infty$.

\begin{prop}\label{prop:an_equivalent_condition} $X$ satisfies \eqref{eq:linearity_condition} if and only if for a.e. $y\in \bbR$
\begin{equation}
\bbE \left[  \sign \left(X - ay  \right)  \phi(y-X) \right] =0,\label{eq:EquivlentCondition} 
\end{equation} 
where $\phi$ denotes the probability density function (pdf) of a standard Gaussian random variable.
\end{prop}
\begin{proof}
We seek to show that for $f(Y)=aY$, the condition in \eqref{eq:Orthogonality_like_property} is equivalent to  \eqref{eq:EquivlentCondition}. 
Note that \eqref{eq:EquivlentCondition} can be equivalently re-written as: for all $\eta$ such that  $ \bbE[ |\eta(Y)|]<\infty$
\begin{align}
0&=\bbE \left[ \sign \left(X - aY  \right) \eta(Y) \right]\\
&=\bbE \left[  \bbE \left[\sign \left(X - aY  \right) |Y \right] \eta(Y) \right]\\
&=\bbE \left[ h(Y) \eta(Y) \right], \label{eq:inner_product_zero_like_condition}
\end{align} 
where we have defined $h(y)= \bbE \left[\sign \left(X - aY  \right) |Y=y \right]$.  The fact that \eqref{eq:inner_product_zero_like_condition} is equivalent to 
\begin{equation}
0=h(y) \text{ a.e. } y \in \mathbb{R},
\end{equation} 
is a standard fact (see, for example, \cite[Lem.~10.1.1]{resnick2019probability}). 
This concludes the proof. 
\end{proof}

We now show that  \eqref{eq:EquivlentCondition}  can be restated as a convolution problem.  
\begin{prop}\label{prop:Equivalent_convolution_problem} $ P_X$ satisfies \eqref{eq:linearity_condition} if and only if for all $y\in \bbR$
\begin{equation}
0 = \int_{-\infty}^\infty   g'(x-y)   \,  \rmd \mu ( x) \label{eq:Equivalent_COnvolution_Problem}
\end{equation} 
where   we let   
 \begin{align}
    \rmd \mu(  x) &= \exp\left( (1-a) \frac{x^2}{2} \right)  \rmd P_X ( \sqrt{a} \,  x),\\
 g(x) &= \min(\Phi(x), 1-\Phi(x) ), 
\end{align} 
and where $\Phi$ is the cdf of a standard Gaussian random variable.  
\end{prop} 
\begin{proof}
Observe the following sequence of implications. Starting with \eqref{eq:EquivlentCondition}
\begin{align}
&0 =  \int \sign \left( x- ay  \right)  \phi(y-x) \, \rmd P_X(x)\\
&\Leftrightarrow 0 =  \int \sign \left( \frac{x}{\sqrt{a}}- \sqrt{a} y  \right)  \phi(y-x) \,  \rmd P_X( x)\\
&\Leftrightarrow 0 =  \int \sign \left(x- \sqrt{a} y  \right)  \phi(y- \sqrt{a} x) \, \rmd P_X( \sqrt{a}   x)\\
&\Leftrightarrow 0 =  \int \sign \left(x- y  \right)  \phi \left( \frac{y}{\sqrt{a}}- \sqrt{a} x \right) \,  \rmd P_X( \sqrt{a}  x)\\
&\Leftrightarrow 0 =  \int \sign \left(x- y  \right)  \rme^{ xy  -\frac{x^2}{2}}  \rme^{ (1-a) \frac{x^2}{2} } \, \rmd  P_X( \sqrt{a}  x)\\
& \Leftrightarrow  0 = \int   \sign(x-y) \phi(y-x)   \,  \rmd \mu (x). 
\end{align} 
 To show the second representation let
\begin{equation}
g(x)  = \min(\Phi(x), 1-\Phi(x) ) ,
\end{equation}
and note that
\begin{equation}
g'(x)  = -\sign(x) \phi(x).
\end{equation}
This concludes the proof. 
\end{proof} 

 At this point, we have reduced the problem to solving a convolution equation, and so it is natural to consider a Fourier approach. However, caution must be exercised regarding the validity of such an approach. In particular, we need to be able to make sense of the Fourier transform of $\mu$, which in general can grow super-exponentially and therefore may not even be a tempered distribution. In order to proceed, we will show that if $\mu$ satisfies \eqref{eq:Equivalent_COnvolution_Problem} and $\mu$ is a \emph{non-negative} measure, then it does indeed need to be sufficiently well-behaved to have a Fourier transform, and in particular needs to be a tempered distribution. Before getting to our main results, in the next subsection we show that the non-negativity of $\mu$ is critical to this argument, and without it non-trivial counterexamples can be found.

To conclude this subsection, we will present the Fourier transform of 
$g'$, which will be useful in our main proof.
\begin{lem}\label{lem:Dawson_fun} Let $\widehat{g'}$ denote the Fourier transform\footnote{We use the following convention for the Fourier transform: $ \frac{1}{\sqrt{2 \pi}} \int_{-\infty}^\infty f(x) e^{-j \omega x} \rmd x $.} 
of $g'$. We have that
\begin{equation}
\widehat{g'}(\omega) = -\frac{2 j}{ \sqrt{\pi}} D \left( \frac{\omega}{\sqrt{2}} \right)
\end{equation}
where  $j =\sqrt{-1}$ and $D(\omega)$ is the \emph{Dawson function}  defined as 
\begin{equation}
D(\omega) = \rme^{-\omega^2} \int_0^\omega  \rme^{t^2} \rmd t. 
\end{equation}
\end{lem}
\begin{proof}
This is a standard result that can, for example, be found in \cite{abramowitz1964handbook}. 
\end{proof}

\subsection{Operator Theory Perspective and Why the Positivity Assumption Is Important} 
\label{sec:operator_persepective} 
 
Consider the following integral operator on the set of $L^1$ functions:
\begin{equation}
T_a[f](y)= \int_{-\infty}^\infty K_a(x,y) f(x) \rmd x \label{eq:Operator_T_def}
\end{equation} 
where the \emph{kernel}  $ K_a(x,y)$ is given by
\begin{equation}
 K_a(x,y)=\sign(x-ay) \phi(y-x) . 
\end{equation}
In this section, we take an operator theory perspective and study the null-space of $T_a$. 

If we restrict our attention only to random variables $X$ having a pdf, finding the set of solutions to \eqref{eq:EquivlentCondition} is equivalent to characterizing the null space of $T_a[f]$ over the space  $L_+^1 = \{f: f \ge 0, \int f(x) \rmd x <\infty \}$  (i.e., non-negative $L_1$ functions); that is 
\begin{equation}
\mathcal{N}(T_a)= \left \{   f \in L_+^1:  T_a[f]=0 \right \} .
\end{equation} 
In this work, we show that
\begin{equation}
\mathcal{N}(T_a)=\{c\phi_{ \frac{a}{1-a}} :  c\geq 0  \}, 
\end{equation} 
where $\phi_\frac{a}{1-a}$ is Gaussian density with variance $\frac{a}{1-a}$. 

One sensible approach to showing that the Gaussian function $\phi_{ \frac{a}{1-a}}  $ is the only non-trivial solution is to relax the non-negativity constraint on  $f$  and consider a null-space over  $L^1(\bbR)$, that is
\begin{equation}
\mathcal{N}_{L^1}(T_a)= \left \{   f \in L^1(\bbR):  T_a[f]=0 \right \} .
\end{equation} 
Somewhat surprisingly, we show that this $\mathcal{N}_{L^1}(T_a)$ is infinite-dimensional.  

To aid this discussion, we require to understand how the \emph{Gabor wavelet} \cite{shen2006review} is transformed by the operator $T_a$.  Recall that the Gabor wavelet is defined as
\begin{equation}
f_{\mu, \sigma^2, \omega}(x)= \exp \left(-  \frac{(x-\mu)^2}{2 \sigma^2} \right)  \rme^{j x w}, \, x,\ \in \bbR. 
\end{equation} 

\begin{thm} \label{thm:passing_gabbor_wavelet} Assume that $-1 < \sigma^2<\infty$.  Then,\footnote{For $z \in \bbC$, the error function is defined as $\erf(z)= \frac{2}{\sqrt{\pi} } \int_{0}^z \rme^{-x^2} \rmd x$.}
\begin{align}
T_a [f_{\mu, \sigma^2, \omega} ](y) &= 2 \sqrt{ \frac{\pi/2}{b} } \phi(y)   \rme^{-\frac{\mu^2}{2 \sigma^2} }   \rme^{  \frac{(y+\frac{\mu}{\sigma^2}+jw)^2}{2b} } \notag\\
& \quad \cdot \erf \left(  \frac{ (1-ba)y+\frac{\mu}{\sigma^2}+jw}{  \sqrt{2 b} }\right),
\end{align}
where    $b=1+\frac{1}{\sigma^2}$. Moreover, if  $\sigma^2=\frac{a}{1-a}$, then 
\begin{equation}
T_a [f_{\mu, \sigma^2, \omega} ](y)=  c(b,\omega,\mu) \rme^{  - \left (1-a \right)\frac{(y-\mu)^2}{2} } \rme^{ jw ay }, \label{eq:Gabbo_matched_Case}
\end{equation} 
where 
\begin{equation}
c(b,\omega,\mu)=   \sqrt{ a }   \rme^{- \frac{a  \omega^2}{2 }}  \erf \left(  \frac{ \frac{\mu}{\sigma^2}+jw}{  \sqrt{2  \frac{1}{a} } }\right) \rme^{ \frac{jw \frac{\mu}{\sigma^2}}{ \frac{1}{a}} } .
\end{equation}
\end{thm} 
\begin{proof} See Appendix~\ref{app:thm:passing_gabbor_wavelet}.
\end{proof}

At this point, we recall that $z \mapsto \erf(z)$ has infinitely many zeros \cite{fettis1973complex}. For example, the first three zeros are given by 
\begin{align}
z_1 &\approx 1.45061 61632 + j  1.88094 30002,\\
z_2 &\approx 2.24465 92738 + j  2.61657 51407,\\
z_3 &\approx 2.83974 10469 + j 3.17562 80996.
\end{align} 
We also note that due to conjugate symmetry, if $z_n$ is a zero, so are $-z_n, \bar{z}_n$ and $- \bar{z}_n$.  Therefore, in Theorem~\ref{thm:passing_gabbor_wavelet}, by choosing $b=\frac{1}{a}$ and $\frac{ \frac{\mu}{\sigma^2}+jw}{  \sqrt{2  b } }$ to be a zero of the $\erf(z)$ function we arrive at the following result. 

\begin{thm} $\mathcal{N}_{L^1}(T_a)$  is an infinite-dimensional  subset of $L^1(\bbR)$.  Moreover, 
\begin{align}
{\rm span} \left( \bigcup_{ (\mu_n,\omega_n):  z_n=\frac{ \frac{1-a}{\sqrt{a}}  \mu_n + j \sqrt{a} \omega_n }{\sqrt{2}} 
}  \hspace{-0.4cm}f_{\mu_n, \frac{a}{1-a}, \omega_n}  \right) \subseteq \mathcal{N}_{L^1}(T_a), \label{Eq:sub_set_null}
\end{align} 
where the $z_n$'s are the zeros of the $\erf$ function. 
\end{thm} 
\begin{proof}
Chose $\omega_n$ and $\mu_n$ such that 
\begin{align}
z_n & = \frac{ \frac{\mu}{\sigma^2}+jw}{  \sqrt{2  b } }  \\
&= \frac{ \frac{\mu}{  \frac{a}{1-a} }+jw}{  \sqrt{2  \frac{1}{a} } } =\frac{1}{\sqrt{2}}  \left(  \frac{1-a}{\sqrt{a}}  \mu_n + j \sqrt{a} \omega_n \right) 
\end{align}
where $z_n$ is a zero of the $\erf$ function.  Then, by using \eqref{eq:Gabbo_matched_Case}, we have that 
\begin{equation}
 T_a [f_{\mu_n, \frac{a}{1-a}, \omega_n} ] (y) = 0,
\end{equation} 
since $c(b,\omega_n,\mu_n)=  0$. Thus, the collection of $f_{\mu_n, \frac{a}{1-a}, \omega_n} $'s are in the null space of $T_a$. Furthermore, since there are infinitely many such functions, which follows from the fact that $\erf(z)$ function has infinitely many zeros, and since Gabor wavelets are linearly independent (provided that the set does not form too dense of a set of points in the $(\mu,\omega)$-plane)\cite{gabors}, we arrive at a conclusion that $\mathcal{N}_{L^1}(T_a)$ is infinite-dimensional. 
\end{proof} 

The above theorem says that the null space of $T_a$ over  $L^1(\bbR)$ contains infinitely many Gabor wavelets. These Gabor wavelets are special in the sense that the location and frequency components correspond to the real and imaginary parts of zeros of the $\erf$ function, respectively.  Note that the above also implies that the following real-valued functions are also solutions: 
\begin{align}
&\frac{ f_{\mu_n, \frac{a}{1-a}, \omega_n} (x) +  f_{\mu_n, \frac{a}{1-a}, \omega_n} (x)  } {2} \notag\\
& \qquad \qquad =   \exp \left(-  \frac{ \left(x- \mu_n \right)^2}{2 \frac{a}{1-a} } \right)  \cos(x \omega_n), \label{eq:cos_counter_ex}\\
&\frac{ f_{\mu_n, \frac{a}{1-a}, \omega_n} (x) -  f_{\mu_n, \frac{a}{1-a}, \omega_n} (x)  }{2j} \notag\\
&\qquad \qquad =  \exp \left(-  \frac{ \left(x- \mu_n \right)^2}{2  \frac{a}{1-a} } \right)  \sin(x \omega_n).
\end{align}

The above discussion shows that the non-negativity of probability measures plays a crucial role; otherwise, we get infinitely many solutions. Indeed, the non-negativity of the probability measures will also play a key role in our proof.  An interesting implication of our results is that infinite linear combinations of the functions defined on the left-hand side of \eqref{Eq:sub_set_null} can only result in a single non-negative function, which is the Gaussian pdf.
Finally, in Section~\ref{sec:Lp_losses}, we will revisit functions similar to Gabor wavelets, where such functions will be used to demonstrate that infinitely many distributions can induce linearity in Bayesian estimators under other $L^p$ loss.

\section{Main Results} 
\label{sec:Main_Results}

In this section, we present our main result. 
Our main technique uses the theory of \emph{distributions} and \emph{tempered distributions} from functional analysis, 
the background of which can be found in \cite{stein2011functional}.
In a nutshell, this is a rigorous framework for extending the concept of functions to include important examples such as the Heaviside delta function that arise naturally from Fourier analysis.

The main result of this work is the following theorem. 

\begin{thm} \label{thm:main_result}  The conditional median satisfies  $\sfm(X|Y=y) = ay, \, \forall y\in \bbR$ if and only if $a\in[0,1)$ and $X \sim \mathcal{N}\left(0,\frac{a}{1-a} \right)$. 
\end{thm}

\begin{proof}

Recall that according to Proposition~\ref{prop:an_equivalent_condition}, the linearity of the conditional median is equivalent to the following integral equation: 
\begin{equation}
\int_{-\infty}^{\infty}g'(t-x)\rmd \mu(x)=0, 
\quad \forall t\in\mathbb{R}.
\label{e5}
\end{equation}

Assuming that $\mu$ satisfies \eqref{e5}, 
Lemma~\ref{lem_temp} below shows that $\mu$ must be a tempered distribution, and therefore we can take its Fourier transform, denoted as $\widehat{\mu}$. Critically, Lemma~\ref{lem_temp} uses the positivity of $\mu$. 

Next, observe that $\widehat{\mu}$ is a (tempered) distribution supported at the origin,
as is shown in
Lemma~\ref{lem:Four_mu_origin} below.
Equivalently, we have that $\mu$ can be represented as a polynomial function
\begin{equation} \mu(x)=\sum_{i=0}^ka_ix^i.
\end{equation}

We next show that all coefficients but $a_0$ are zero and that $\mu$ is  a constant.  
Suppose that $\mu(t-x)=\sum_{i,j\colon i+j\le k}b_{ij}t^ix^j$ for some coefficients $(b_{ij})$. Evidently, $b_{ij}=0$ if and only if $a_{i+j}=0$.
Next, define
\begin{align}
\epsilon_j:&= \int g'(x)x^j = -\bbE[ \sign(Z) 
 Z^j] \notag\\
 &= \left \{ \begin{array}{cc}  
 0  & j \text{ even}\\
 -\bbE[|Z|^j] & j \text{ odd}. 
 \end{array} \right. 
\end{align}
Therefore, the condition for the linearity in \eqref{e5} can be written as: for all $t \in \bbR$
\begin{equation}
0 = \int_{\infty}^{\infty}g'(t-x)\rmd \mu(x)=\sum_{i+j\le k,\,2\nmid j}
b_{ij}t^i\epsilon_j,
\end{equation}
 where $2\nmid j$ denotes that $j$ is not divisible by $2$. 
Clearly, a polynomial of degree $k$ is zero on the real line if and only if all coefficients are zero. Thus, we have that $b_{ij} =0$ for all $i \le k$ and all  odd $j \le k$. This implies that all $a_i=0$ for $i  \ge 1 $. This shows that $\mu(x) =a_0$. Consequently, using the definition of $\mu$, we have that 
\begin{equation}
a_0 = \exp\left( (1-a) \frac{x^2}{2} \right) \rmd P_X ( \sqrt{a}   x)
\end{equation}
which implies that $\rmd P_X( x)  \propto  \exp\left( -\frac{(1-a)}{a} \frac{x^2}{2} \right) $. By Theorem \ref{thm:admissable_a}, we can restrict our attention to $a\in [0,1]$, and for $a=1$, this solution does not give a proper probability distribution. This concludes the proof.
\end{proof}

\begin{rem}
A key component in the proof of Theorem~\ref{thm:main_result} was to show that $\mu$ is a tempered distribution (Lemma~\ref{lem_temp}),
which relies crucially on the positivity of the measure $\mu$ in establishing the growth estimate~\eqref{e_growth}.
Indeed, in the counterexamples from Section \ref{sec:operator_persepective} above, we do not have positivity, and  the $\mu$ corresponding to \eqref{eq:cos_counter_ex} ends up being represented by
\begin{align}
&\mu(x) \notag\\
& = \exp\left(\frac{1-a}{2}x^2\right)f(\sqrt{a}x) \\
& = \exp\left(\frac{1-a}{2}x^2\right)\exp\left(-\frac{(1-a)(\sqrt{a}x-\mu_n)^2}{2a}\right) \notag\\ 
&\qquad \cdot \cos(\omega_n \sqrt{a} x) \\
& = C\exp\left(\frac{1-a}{\sqrt{a}}x\mu_n\right)\cos(\omega_n \sqrt{a} x)
\end{align}
which is not a tempered distribution. 
Once $\mu$ is shown to be a tempered distribution, Fourier transform techniques can be applied. 
\end{rem}

\begin{lem}\label{lem_temp}
    If a non-negative measure $\mu$ satisfies \eqref{e5}, then $\mu$ is a tempered distribution. 
\end{lem}
\begin{proof}
    
Since $\max\{|g'(x)|,g(x)\}\lesssim e^{-\frac{x^2}{2}}$, it is easy to check that 
\begin{equation}
G(t):=\int_{-\infty}^{\infty}g(t-x)\rmd \mu(x)
\end{equation}
is convergent for any $t$, and by dominated convergence and \eqref{e5}, we have that $G'(t)=0$.
This shows that $G(t)=C$ is a constant.
Since $g(\cdot)\ge c 1_{[-1,1]}$ for some $c>0$, and $\mu$ is a non-negative measure, it follows that $\mu([t-1,t+1])\le \frac{C}{c}$ for all $t$.
In particular, there exists a constant $C_1$ such that  $\mu([-R,R])\le  C_1R$ for all $R>1$.
Then for any $\psi$ supported on $\{|x|\le R\}$, we obtain 
\begin{equation}
\int_{-\infty}^\infty \psi(x)\rmd \mu(x)\le C_1R\sup_{|x|\le R}|\psi(x)|.
\label{e_growth}
\end{equation}
This growth estimate implies that $\mu$ is a tempered distribution \cite[p.~147, Exercise~7]{stein2011functional}. This concludes the proof. 
\end{proof}

\begin{lem}\label{lem:Four_mu_origin}
 $\widehat{\mu}=\sum_{i=0}^k a_i \delta^{(i)}$ is a finite sum.
\end{lem}
\begin{proof}
We will show that $\widehat{\mu}$ is supported only at the origin, which using a standard result from \cite[p.~110]{stein2011functional}  implies that  $\widehat{\mu}=\sum_{i=0}^k a_i \delta^{(i)}$ where $k$ is finite.  

Let $\mathcal{D}$ be the set of smooth (infinitely differentiable) functions with compact support, equipped with the topology of convergence of all the (any order of) derivatives and containment of the support.
A distribution is a continuous linear functional on $\mathcal{D}$.
Let $\psi\in\mathcal{D}$ be an arbitrary function supported on $\mathbb{R}\setminus\{0\}$. Using Lemma~\ref{lem:Dawson_fun}, we know that $\widehat{g'}$ is the Dawson function, which is a  smooth function vanishing only at zero, therefore, there  exists $\xi\in\mathcal{D}$ such that $\psi =\widehat{g'}\cdot\xi$ (where $\cdot$ indicates pointwise product of two functions). 
Then
\begin{align}
\widehat{\mu}(\psi)
&=\mu(\widehat{\psi} \,)
\label{e8}
\\
&=\mu(g'^{\sim}*\widehat{\xi} \,)
\\
&=\int
\left(
\int g'(t-x)\widehat{\xi}(t) \rmd t
\right)  \rmd \mu(x)
\\
&=
\int\left(\int
g'(t-x) 
 \rmd \mu(x)
\right)
\widehat{\xi}(t)\rmd t
\label{e11}
\\
&=0,
\label{e12}
\end{align}
where \eqref{e8} uses the definition of Fourier transform of a distribution \cite[p.~108]{stein2011functional},
and $\sim$ indicates the reflection of a function, i.e., $f^{\sim}(x)=f(-x)$ for any $f$;
\eqref{e11} uses Fubini's theorem; and
\eqref{e12} uses \eqref{e5}.
This implies that $\widehat{\mu}$ is supported at $0$ and concludes the proof. 
\end{proof}

\section{Discussion, Extensions and Future Directions} \label{sec:Discussion}
This work has focused on characterizing which prior distributions give an optimal estimator (with respect to $L^1$ loss) that is linear, which is equivalent to the answering the question of when conditional medians are a linear function of the observation.  We have focused on a Gaussian noise model and $L^1$ loss, but the question can be considered more generally. In this section, we will discuss several interesting future directions and show several extensions.  

\subsection{ A Near Miss for Poisson Noise }
One interesting direction is to consider the case of Poisson noise, where the input-output relationship is given by 
\begin{equation}
    P_{Y|X}(y|x) = \frac{1}{y!} x^y \rme^{- x}, \, x \in \bbR_{+}, y \in \bbN_0
\end{equation}
with the convention that  $0^0=1$.

It is well-known that a linear conditional expectation is induced by gamma distribution prior, that is, when  $X \sim {\rm Gam}(\alpha, \beta)$ where the pdf of a gamma distribution is given by \begin{equation}
f_X(x)= \frac{\beta^\alpha}{ \Gamma(\alpha)  }  x^{\alpha-1} \eu^{-\beta x},\, x \ge 0, \label{eq:pdfGamma}
\end{equation}
 where $\alpha>0$ is the shape parameter and $\beta>0$ is the rate parameter. The gamma is a unique such distribution \cite{johnson1957uniqueness, chou2001characterization, dytso2020estimation}. Moreover, the conditional expectation is given by 
 \begin{equation}
 \bbE[X|Y=y]=   \frac{1}{\beta+1}y +\frac{\alpha}{\beta+1}, y  \in \bbN_0. 
 \end{equation}
 and the posterior distribution is also gamma with  
 \begin{equation}
 X \mid Y=y  \sim  {\rm Gam}(\alpha+y, \beta+1), \,  y \in \bbR.  \label{eq:Conditional_mean_poisson}
 \end{equation} 
 Now, the median of the gamma distribution does not have a closed-form and is given by 
 \begin{equation}
 m(X|Y=y)= \frac{1}{\beta} \gamma^{-1} \left(\frac{1}{2}, \alpha + y \right), \label{eq:Conditional_median_poisson}
 \end{equation}
 where $\gamma^{-1}$ is the inverse of the lower incomplete gamma function and needs to be computed numerically.  Clearly, the conditional median, unlike the conditional mean, is \emph{not} linear. Approximations of the median of the gamma distribution have received some attention in the literature, and the interested reader is referred to \cite{chen1986bounds,choi1994medians,berg2006chen}. 
 
Although the median in \eqref{eq:Conditional_median_poisson} is not linear, it is nearly linear.  In fact, the deviation from linearity is  rather small and decreases as $O \left(\frac{1}{y} \right)$   \cite{berg2006chen}. Fig.~\ref{fig:comp_median_to_ mean} compares the conditional mean in \eqref{eq:Conditional_mean_poisson} and the conditional median in \eqref{eq:Conditional_median_poisson} for $(\alpha,\beta) =(1,1)$. 

  \begin{figure*}[t]  
	\begin{subfigure}[t]{0.5\textwidth}
	\centering
\begin{tikzpicture}

\begin{axis}[%
width=7cm,
height=6cm,
at={(1.011in,0.642in)},
scale only axis,
xmin=0,
xmax=20,
xlabel style={font=\color{white!15!black}},
xlabel={$y$},
ymin=0,
ymax=12,
axis background/.style={fill=white},
xmajorgrids,
ymajorgrids,
legend style={legend cell align=left, align=left, draw=white!15!black}
]
\addplot [color=blue, mark=*, mark options={solid, blue}]
  table[row sep=crcr]{%
0	0.346573590279973\\
1	0.83917349500833\\
2	1.33703015686178\\
3	1.83603037442545\\
4	2.33545444139799\\
5	2.83508059435604\\
6	3.33481853727489\\
7	3.8346247212504\\
8	4.33447559218519\\
9	4.83435730735707\\
10	5.33426120191816\\
11	5.83418157652238\\
12	6.33411452936932\\
13	6.83405729967245\\
14	7.3340078791654\\
15	7.83396477215862\\
16	8.33392684186981\\
17	8.83389320889423\\
18	9.33386318228584\\
19	9.83383621165284\\
20	10.3338118531553\\
21	10.833789744916\\
22	11.3337695889393\\
23	11.8337511376132\\
24	12.3337341834942\\
25	12.8337185514812\\
26	13.3337040927493\\
27	13.8336906799998\\
28	14.3336782037047\\
29	14.8336665691106\\
30	15.333655693829\\
31	15.833645505883\\
32	16.3336359421136\\
33	16.8336269468693\\
34	17.3336184709229\\
35	17.8336104705706\\
36	18.3336029068787\\
37	18.8335957450498\\
38	19.3335889538879\\
39	19.8335825053446\\
40	20.3335763741311\\
41	20.8335705373879\\
42	21.3335649743994\\
43	21.8335596663484\\
44	22.3335545961036\\
45	22.8335497480351\\
46	23.3335451078528\\
47	23.8335406624669\\
48	24.3335363998636\\
49	24.8335323089972\\
50	25.3335283796942\\
};
\addlegendentry{Median}

\addplot [color=red, mark=*, mark options={solid, red}]
  table[row sep=crcr]{%
0	0.5\\
1	1\\
2	1.5\\
3	2\\
4	2.5\\
5	3\\
6	3.5\\
7	4\\
8	4.5\\
9	5\\
10	5.5\\
11	6\\
12	6.5\\
13	7\\
14	7.5\\
15	8\\
16	8.5\\
17	9\\
18	9.5\\
19	10\\
20	10.5\\
21	11\\
22	11.5\\
23	12\\
24	12.5\\
25	13\\
26	13.5\\
27	14\\
28	14.5\\
29	15\\
30	15.5\\
31	16\\
32	16.5\\
33	17\\
34	17.5\\
35	18\\
36	18.5\\
37	19\\
38	19.5\\
39	20\\
40	20.5\\
41	21\\
42	21.5\\
43	22\\
44	22.5\\
45	23\\
46	23.5\\
47	24\\
48	24.5\\
49	25\\
50	25.5\\
};
\addlegendentry{Mean}

\end{axis}
\end{tikzpicture}%
	\caption{Conditional mean in \eqref{eq:Conditional_mean_poisson} vs. conditional median in \eqref{eq:Conditional_median_poisson}.  } 
	\end{subfigure}%
 ~
 \begin{subfigure}[t]{0.5\textwidth}
	\centering
\begin{tikzpicture}

\begin{axis}[%
width=7cm,
height=6cm,
at={(1.011in,0.642in)},
scale only axis,
xmin=0,
xmax=20,
xlabel style={font=\color{white!15!black}},
xlabel={$y$},
ymin=-0.168,
ymax=-0.152,
axis background/.style={fill=white},
xmajorgrids,
ymajorgrids,
legend style={legend cell align=left, align=left, draw=white!15!black}
]
\addplot [color=black, thick]
  table[row sep=crcr]{%
0	-0.153426409720027\\
1	-0.16082650499167\\
2	-0.16296984313822\\
3	-0.163969625574552\\
4	-0.164545558602008\\
5	-0.164919405643965\\
6	-0.165181462725114\\
7	-0.165375278749598\\
8	-0.165524407814813\\
9	-0.165642692642934\\
10	-0.165738798081843\\
11	-0.165818423477616\\
12	-0.165885470630682\\
13	-0.165942700327551\\
14	-0.165992120834603\\
15	-0.166035227841375\\
16	-0.166073158130191\\
17	-0.166106791105769\\
18	-0.166136817714158\\
19	-0.166163788347164\\
20	-0.16618814684475\\
21	-0.166210255084026\\
22	-0.166230411060683\\
23	-0.166248862386841\\
24	-0.16626581650579\\
25	-0.166281448518792\\
26	-0.166295907250744\\
27	-0.166309320000241\\
28	-0.166321796295307\\
29	-0.166333430889384\\
30	-0.166344306171041\\
31	-0.166354494117009\\
32	-0.166364057886394\\
33	-0.166373053130716\\
34	-0.166381529077132\\
35	-0.166389529429395\\
36	-0.166397093121322\\
37	-0.16640425495024\\
38	-0.166411046112071\\
39	-0.166417494655448\\
40	-0.166423625868884\\
41	-0.166429462612054\\
42	-0.166435025600578\\
43	-0.1664403336516\\
44	-0.166445403896397\\
45	-0.166450251964935\\
46	-0.166454892147158\\
47	-0.166459337533055\\
48	-0.166463600136353\\
49	-0.166467691002818\\
50	-0.166471620305813\\
};

\end{axis}
\end{tikzpicture}%
	\caption{Difference between \eqref{eq:Conditional_mean_poisson} vs.   \eqref{eq:Conditional_median_poisson}.  } 
	\end{subfigure}%

 \caption{Poisson Noise Case: Conditional mean vs. conditional median under a gamma prior with $(\alpha,\beta)= (1,1)$.}
  \label{fig:comp_median_to_ mean}
\end{figure*}
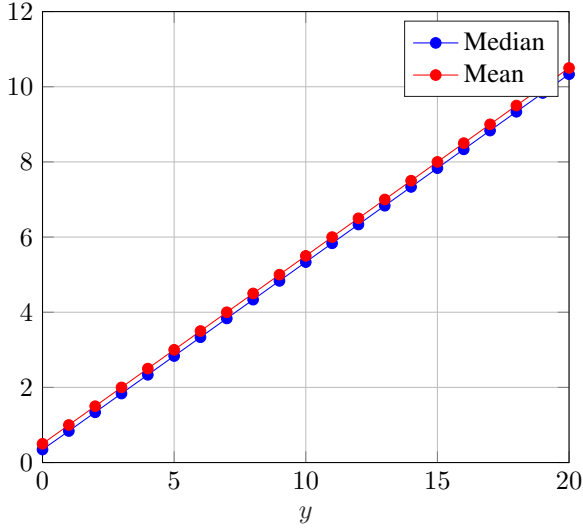
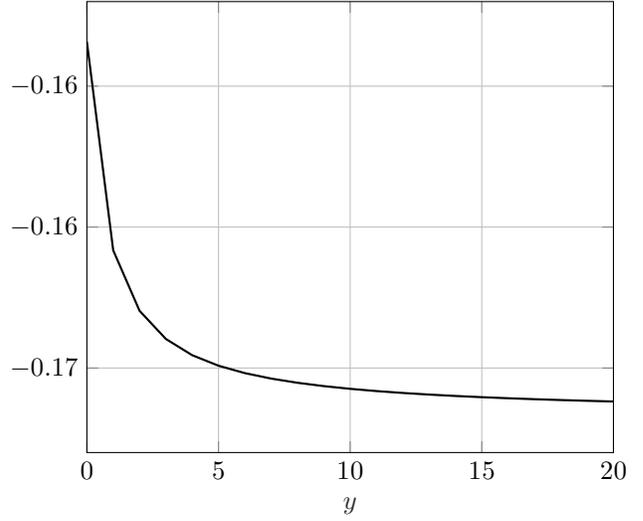 

An interesting future direction will be to see if there exists another prior on $X$ that induces linearity of the conditional median. It is not difficult to see, following the same proof as in Proposition~\ref{prop:an_equivalent_condition}, that $f_X$ induced linearity for a give  pair $(a,b)$\footnote{Note that since $X$ is only supported on non-negative values, in order not to lose generality, we need to consider affine estimators $ay+b$ instead of just a linear estimator $ay$.} if and only if
\begin{equation}
0 = \int_0^\infty  \sign \left(x - ay -b \right)  x^y \rme^{-x} f_X(x) \rmd x, \,   y \in  \bbN_0. \label{eq:suf_for_poisson}
\end{equation}
 There are a few key features distinguishing \eqref{eq:suf_for_poisson} from the Gaussian case. Firstly, integration is confined to non-negative values, unlike in the Gaussian scenario. Secondly, the integral in \eqref{eq:suf_for_poisson} must equate to zero for non-negative integers rather than the entire real line, potentially imposing extra constraints on the solution methodology.

\subsection{On Other $L^p$ Losses}
\label{sec:Lp_losses}

In this section, we consider similar questions to that of when the median is linear for other  $L_p$ losses. More precisely, we consider a Bayesian risk of the form: $p \ge 1$
\begin{equation}
\inf_{f} \bbE[ |  X - f(Y) |^p]
\end{equation}
We again are interested in finding the distributions that would lead to the optimality of linear estimators.  
The condition for linearity akin to the one in \eqref{eq:EquivlentCondition} for the $L^p$ losses with even $p$ is given by: for $y \in \bbR $
\begin{align}
\int_{-\infty}^\infty \sign(x- ay) | x - ay |^{p-1} \phi(y-x)  \rmd P_X(x) = 0 ,
\end{align}
which, by following the steps in Proposition~\ref{prop:Equivalent_convolution_problem}, can be re-written as a convolution: for $y \in \bbR$ 
\begin{align}
\int_{-\infty}^\infty \sign(x- y) | x - ay |^{p-1} \phi(y-x)  \rmd \mu(x) = 0 ,  \label{eq:condition_for_lp}
\end{align}
where as before $ \rmd \mu(x) = \exp\left( (1-a) \frac{x^2}{2} \right)  \rmd P_X(\sqrt{a} \,  x) $.

Next, somewhat surprisingly, we show that for $p>2$ there are infinitely many priors that induce linearity.  
\begin{thm}
    \label{thm13}
   Fix a $ p \in (2,\infty)$. Then for every $|\rho| \le 1$ and $\theta \in \bbR$, there exists an $\omega$ such that the density
   \begin{equation}
f_X(x) \propto \rme^{- \frac{1-a}{a} \frac{x^2}{2}} \left(1 + \rho \cos \left( \frac{\omega x}{\sqrt{a}} +\theta \right) \right) \label{eq:lp_densities}
\end{equation}
induces a linear minimum $L^p$ estimator. Moreover,  for even $p$, 
 $\omega$'s are given  by the zeros of the probabilist's Hermite polynomial $H_{e_{p-1}}$. 
\end{thm}
\begin{proof} 
We show that there is an appropriate choice of $\omega$ such that the density in \eqref{eq:lp_densities} satisfies \eqref{eq:condition_for_lp} , which would imply that the above density induces linearity of the conditional $L^p$ estimator.  We have that 
\begin{align}
&\int_{-\infty}^\infty \sign(x-y) | x - y |^{p-1} \phi(y-x) \notag\\
& \qquad \cdot \exp\left( (1-a) \frac{x^2}{2} \right) f_X(\sqrt{a}x)  \rmd x \notag\\
& =  \int_{-\infty}^\infty \sign(x-y) | x - y |^{p-1} \phi(y-x) \notag\\
& \qquad \cdot \left(1 +\rho\cos \left( \omega x +\theta\right) \right)  \rmd x\\
& = \rho \int_{-\infty}^\infty \sign(x-y) | x - y |^{p-1} \phi(y-x) \notag\\
& \qquad \cdot \cos \left( \omega x +\theta \right)   \rmd x \label{eq:using_that_p_is_even} \\
& = \rho \mathsf{Re} \left \{ \rme^{-j \omega y +j \theta} \mathcal{F} \left( \sign(\cdot) | \cdot  |^{p-1} \phi(\cdot)  \right)(\omega)\right \} 
\end{align}
where  \eqref{eq:using_that_p_is_even} follows from the fact that the function is odd.

For even $p$, the proof is simple and 
\begin{align}
&\mathsf{Re} \left \{ \rme^{-j \omega y +j \theta} \mathcal{F} \left(  ( \cdot  )^{p-1} \phi(\cdot)  \right)(\omega)\right \}  \notag\\
& = \mathsf{Re} \left \{ \rme^{-j \omega y  +j \theta}  \frac{\rmd^{p-1}}{ \rmd \omega^{p-1}}\phi(\omega)\right \}\\
& = \mathsf{Re} \left \{ \rme^{-j \omega y  +j \theta } (-1)^{p-1} H_{e_{p-1}}(\omega)\phi(\omega)\right \}, \label{eq:using_heremite}
\end{align}
where \eqref{eq:using_heremite} follows by using the identity between derivative of the Gaussian density and the probabilist's Hermite polynomials $H_{e_k}$. Note that $H_{e_{p-1}}$ has exactly $p-1$ zeros, thus placing $\omega$  at any of these locations will result in \eqref{eq:using_heremite} being equal to zero. 

For the general $p$, we require to show that 
\begin{equation}
    \mathcal{F} \left( \sign(\cdot) | \cdot  |^{p-1} \phi(\cdot)  \right)(\omega)
\end{equation}
has nonzero roots.  The proof of this fact is shown in Appendix~\ref{app:thm_gen_lp_case}. This concludes the proof. 
\end{proof}  

Fig.~\ref{fig:example_lp_dens} shows a few examples of the distributions in \eqref{eq:lp_densities} for $p=4$ where we note that  $H_{e_3}$ has zeros at $ \{ \pm \sqrt{3}, 0 \}$. An interesting observation to note is that a non-symmetric distribution induces linearity of the estimator. 

\begin{figure*}[t]
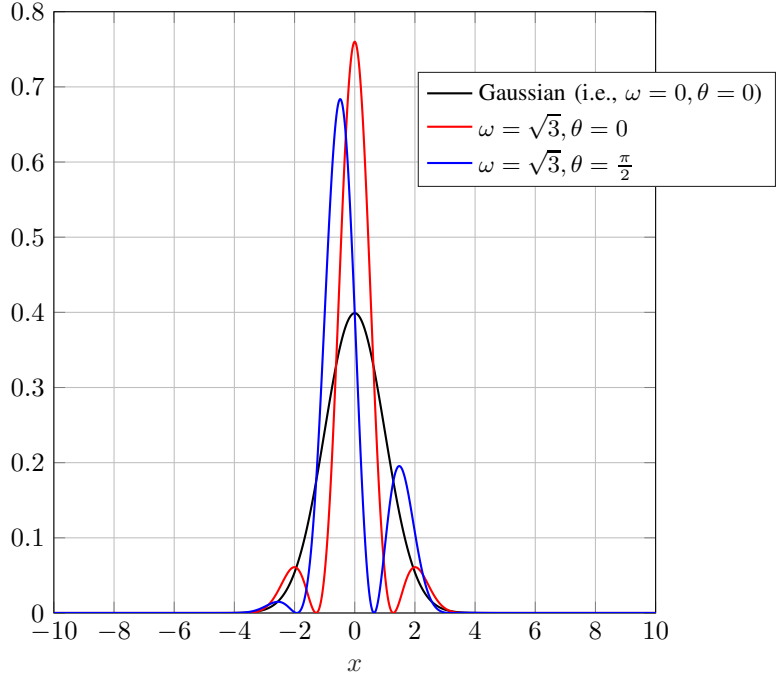
  
	\centering
%
	\caption{Example of probability densities in \eqref{eq:lp_densities} for $p=4$ and $\rho=1$.}
  \label{fig:example_lp_dens}
\end{figure*} 

We note that the authors of \cite{akyol2012conditions} have shown that under the assumption that $Z$ and $X$ have the same variance, the estimator is linear if and only if $X$ is Gaussian.\footnote{In fact the result of \cite{akyol2012conditions} holds for arbitrary distributions on the noise $Z$, and that that if $X$ and $Z$ have the same variance then the estimator is linear if and only if $X$ and $Z$ have the same distribution.}  Our results do not contradict those of \cite{akyol2012conditions} since it is not difficult to demonstrate that the variance of the distribution in \eqref{eq:lp_densities} for $\theta = 0$ is given by 
\begin{equation}
    \mathsf{Var}(X) = \frac{a}{1-a} \frac{1 + \left(1-\frac{\omega^2}{1-a} \right) \rho \rme^{-\frac{\omega^2}{2(1-a)} }}{ 1 +  \rho \rme^{-\frac{\omega^2}{2(1-a)} }},
\end{equation}
which is 
 only equal to one  if $(a,\omega) =(\frac{1}{2}, 0)$ and/or $(a,\rho) = (\frac{1}{2}, 0)$.

 We complement the result in Theorem~\ref{thm13} by showing that for $p\in [1,2]$, the Gaussian density is only solution.
 \begin{thm}\label{thm_gen_lp_case}
For any $p\in [1,2]$, Gaussian priors are the only ones inducing linear optimal estimators.
\end{thm}
\begin{proof}
    The proof is given in Appendix~\ref{app:thm_gen_lp_case}. 
\end{proof}

\subsection{Extension to the Natural Exponential Family } 
In this section, we extend our results to members of certain exponential families of distributions. Recall that a \emph{natural exponential family}, parameterized by $x \in \bbR$,  is characterized by the pdf of the form
\begin{equation}
f_{Y|X}(y|x) = h(y) \rme^{ xy - \psi(x)}, \, y \in \bbR, \, x\in \bbR, \label{eq:Expo_family}
\end{equation}
where $h(y): \bbR \to [0,\infty)$ is known as the \emph{base measure} and $\psi:\bbR \mapsto \bbR $ is known as the \emph{log-partition function}.   The log-partition function is a convex function, and its  convex conjugate  is  defined as
\begin{equation}
\psi^*(y) = \sup_{x} \left( x y  -\psi(x) \right), \,  y \in \bbR. 
\end{equation}

We now seek to characterize the prior that induces linearity of  the conditional median of $X$ when the noise distribution induces a conditional distribution for $Y$ given $X = x$ whose pdf is a member of a natural exponential family   in \eqref{eq:Expo_family}.

Importantly, we note that prior literature on the exponential family, such as the canonical result of \cite{diaconis1979conjugate}, has typically focused on related but different questions. Specifically, the focus of \cite{diaconis1979conjugate} and related works is on the linearity of Bayesian estimators of $\psi'(X)$; the motivation being that $\psi'(X)= \mathbb{E}[Y|X]$ (i.e., the mean of the exponential model). We, however, focus on the linearity of the Bayesian estimator of the natural parameter $X$ and not the mean parameter $\psi'(X)$. The two questions coincide only for the Gaussian case when $\psi(x) = \frac{x^2}{2}$. This alternative perspective, naturally, results in a different structure of our conjugate prior than the one derived in \cite{diaconis1979conjugate}. 

For the natural exponential family, the condition for linearity of the conditional median of $X$ is derived verbatim to \eqref{eq:EquivlentCondition} and is given by: for $a \in [0,1)$ 
\begin{equation}
\bbE \left[  {\rm sign}(X - ay)  h(y) e^{  Xy  -\psi(X)  } \right] = 0, \,  \forall y  \in \bbR. \label{eq:expo_family_cond}
\end{equation}
Our restriction to the $L^1$ case here is done for ease of presentation. By using similar techniques to those in Section~\ref{sec:Lp_losses} the results can  be extended to $L^p$ with $1\le p\le 2$. 

While we do not have an answer for every exponential family, we show that our main result can be adapted to find the solution to \eqref{eq:expo_family_cond} when $\psi^*$ is finite. This answers the question for many \emph{natural} exponential families, but answering this in full generality for any exponential family is outside of the scope of this work and would need to handle cases such as \eqref{eq:suf_for_poisson} discussed above.
\begin{thm}  Suppose that the convex conjugate  $\psi^{*}(y)$ is finite for all $ y \in \bbR$. Then, 
the only solution to \eqref{eq:expo_family_cond} is given by 
 \begin{equation}
f_X(x) \propto \rme^{ - \frac{x^2}{2 a} +\psi(x)}, \label{eq:Sol_For_expo}
\end{equation}
provided that it is integrable. 
\end{thm}
\begin{proof} 
Suppose that a probability measure  $P_X$ satisfies  \eqref{eq:expo_family_cond} and induces linearity.  Starting with \eqref{eq:expo_family_cond},  note the following transformations: 
\begin{align}
&\int \hspace{-0.1cm} {\rm sign}(x - ay)  h(y) \rme^{  xy  -\psi(x)  }  \, \rmd P_X(x) = 0, \forall y  \notag\\
&\Longleftrightarrow
\int \hspace{-0.1cm} {\rm sign}(x - ay)   \rme^{-\frac{x^2}{2}  + xy  + \frac{x^2}{2} -\psi(x)  }  \rmd P_X(x) = 0, \forall y \\
&\Longleftrightarrow
\int \hspace{-0.1cm} {\rm sign}(x - ay)   \rme^{ - \frac{(y-x)^2}{2}  } \rme^{\frac{x^2}{2} -\psi(x)}\,  \rmd P_X(x) = 0, \forall y \\
&\Longleftrightarrow
\int g'(y-x) \,\rmd  \nu( x) = 0, \forall y 
\end{align}
where in the last implication we have used the same steps leading to the proof of Proposition~\ref{prop:Equivalent_convolution_problem},  and where we have defined
\begin{align}
\rmd  \nu ( x) &=  \rme^{\frac{x^2}{2} -\psi( \sqrt{a} x)} \rmd P_X( \sqrt{a} x),\label{eq:rescaling_P_X}\\
g'(x) &= -\sign(x) \phi(x).
\end{align}
To show that $\nu$ is tempered distribution, we can use the same technique as in Lemma~\ref{lem_temp}, which requires dominated convergence. To that end, note that 
since $\max\{|g'(x)|,g(x)\}\lesssim e^{-\frac{x^2}{2}}$, we have to show that the following quantity is finite for all $y $:
\begin{align}
&\int \rme^{-\frac{(x-y)^2}{2}} \,\rmd  \nu( x)  \notag\\
& = \rme^{-\frac{y^2}{2}}\int \rme^{xy -\psi( \sqrt{a} x)} \rmd P_X( \sqrt{a}x )  \\
& \le \exp \left( -\frac{y^2}{2}+\sup_{x}  \left( \frac{ x y}{\sqrt{a}} -\psi(  x) \right)  \right) \\
& = \exp \left( -\frac{y^2}{2} + \psi^{*} \left( \frac{  y}{\sqrt{a}} \right)   \right) \; ,
\end{align}
which is finite by assumption. 
Now mirroring the proof in Theorem~\ref{thm:main_result}, we arrive at
\begin{equation}
\rmd P_X(x) \propto  \exp \left( -\frac{x^2}{2 a}  +\psi(x) \right)
\end{equation}
which is a valid probability density only if it is integrable.  This concludes the proof.  
\end{proof}

\subsection{Extension to Higher Dimensions}

Interestingly, the techniques of characterizing prior distribution that induce linearity of the conditional expectation are largely independent of the dimension of the random parameter $X$ to be estimated. For example, the results for the continuous exponential family in the previous section hold verbatim for the case when $X$ is an $n$-dimensional vector. Similar results also hold for the discrete vector cases, such the vector Poisson case \cite{VectorPoisson}. 

 The situation with extending the results from the present paper to the multivariate case is more complex. First of all, unlike for the conditional expectation, there is no unique way of defining the median in the multivariate setting, and several competing definitions exist; the interested reader is referred to \cite{small1990survey} and references therein.  Second of all, the median that minimizes the  $L^1$ loss, also known as the \emph{spatial median} \cite{milasevic1987uniqueness}, that is
\begin{equation}
\mathsf{m}_\text{S}(X|Y)  =\arg \min_{f} \bbE \left[ \| X - f(Y)\| \right],
\end{equation}
does not have a closed-form characterization, unlike the conditional mean, which does have an integral representation.  Along these lines, an interesting future direction is to consider a scenario where $$Y =  X + Z$$ where $Z$ multivariate normal with zero mean and covariance matrix $\sfK$ and independent of the vector $X$, and  consider the following multivariate estimation problem with $L_{p,k}$ loss: for $k\ge 1, \, p \ge 1 $
\begin{equation}
\mathsf{m}_{p,k}(X|Y)= \arg \min_{f} \bbE \left[ \| X - f(Y)\|_{\ell_k}^p \right],  \label{eq:Lp_lq_risk}
\end{equation}
where 
$\|  u \|_{\ell_k} = \left( \sum_{i=1}^n |u_i|^k \right)^{ \frac{1}{k} },   u \in \bbR^n $
is the usual $\ell_k$ norm.  This is a natural estimator in the high-dimensional case. For example, 
the asymptotics of the Bayesian risk  
 in \eqref{eq:Lp_lq_risk} have been considered in \cite{donoho1994minimax}.
Under this setup, one could seek to understand under what conditions on $X$ do we have that 
\begin{equation}
\mathsf{m}_{p,k}(X|Y) = \sfA Y
\end{equation}
where $\sfA$ is a matrix. In particular, with the tools developed in Section~\ref{sec:Main_Results} and Section~\ref{sec:Lp_losses}, it may be possible to characterize the region of $(p,k)$ values such that Gaussian is the only prior that induces linearity.  Recently, in \cite{BDLP_ISIT2024}, we have provided a partial solution for the special case of $k=p$, 
where a result similar to the scalar setting holds:
multivariate Gaussian is the unique prior that induces a linear and positive definite optimal
Bayesian estimator, 
if and only if $p\le 2$.
The case of $k\neq p$, however, remains open.

\section*{Acknowledgements} 
The authors would like to thank Dr.~Tang~Liu for a helpful discussion on the near-miss example in the Poisson case and for pointing out reference \cite{chen1986bounds}.

\appendices

\section{Proof of Theorem~\ref{thm:passing_gabbor_wavelet} } 
 \label{app:thm:passing_gabbor_wavelet}
 We first note that 
\begin{align}
&\int_{-\infty}^\infty  \sign(x-ay) \phi(y-x) f(x) \rmd x \notag\\
&\quad =\int_{-\infty}^{ay} \phi(y-x) f(x) \rmd x-\int_{ay}^\infty \phi(y-x) f(x) \rmd x.
\end{align} 
Next, we will need the following indefinite integral: for any $d$ and $b$
\begin{align}
&\int \exp \left( d x -\frac{bx^2}{2} \right) \rmd x \notag\\
&\quad=\sqrt{ \frac{\pi/2}{b} } \exp \left( \frac{d^2}{2b } \right) {\rm erf} \left( \frac{bx-d}{\sqrt{2b}} \right).
\end{align} 
Next, let $b=1+\frac{1}{\sigma^2}$ and $d=y+\frac{\mu}{\sigma^2}+jw$   and note that 
\begin{align}
&\int_{ay}^\infty \phi(y-x)  \exp \left(-  \frac{(x-\mu)^2}{2 \sigma^2} \right)  \rme^{j x w}\rmd x \notag\\
&=\phi(y)  \rme^{-\frac{\mu^2}{2 \sigma^2}} \int_{ay}^\infty  \rme^{ (y+\frac{\mu}{\sigma^2}+jw) x- \frac{b x^2}{2}} \rmd x\\
&=\phi(y) \rme^{-\frac{\mu^2}{2 \sigma^2}}  \sqrt{ \frac{\pi/2}{b} } \exp \left(  \frac{(y+\frac{\mu}{\sigma^2}+jw)^2}{2b} \right) \notag\\
&\qquad \cdot \left(1-\erf \left(  \frac{b ay -(y+\frac{\mu}{\sigma^2}+jw)}{  \sqrt{2 b} }\right) \right).
\end{align} 
Also,
\begin{align}
&\int_{-\infty}^{ay} \phi(y-x) \exp \left(-  \frac{(x-\mu)^2}{2 \sigma^2} \right)  \rme^{j x w} \rmd x \\
&=\phi(y)  \rme^{-\frac{\mu^2}{2 \sigma^2}} \int_{-\infty}^{ay}  \rme^{ (y+\frac{\mu}{\sigma^2}+jw) x- \frac{b x^2}{2}} \rmd x\\
&=\phi(y) \rme^{-\frac{\mu^2}{2 \sigma^2}}  \sqrt{ \frac{\pi/2}{b} } \exp \left(  \frac{(y+\frac{\mu}{\sigma^2}+jw)^2}{2b} \right) \notag\\
&\qquad \cdot \left(\erf \left(  \frac{b ay -(y+\frac{\mu}{\sigma^2}+jw)}{  \sqrt{2 b} }\right) +1\right)
\end{align} 
Combining everything we have that
\begin{align}
&\int_{-\infty}^\infty  \sign(x-ay) \phi(y-x) \exp \left(-  \frac{(x-\mu)^2}{2 \sigma^2} \right)  \rme^{j x w} \rmd x \notag\\
&= 2 \phi(y)   \rme^{-\frac{\mu^2}{2 \sigma^2}} \sqrt{ \frac{\pi/2}{b} } \exp \left(  \frac{(y+\frac{\mu}{\sigma^2}+jw)^2}{2b} \right) \notag\\
& \quad \cdot \erf \left(  \frac{ (y+\frac{\mu}{\sigma^2}+jw)-bay}{  \sqrt{2 b} }\right). 
\end{align}

Now, for $a=\frac{1}{b}$
\begin{align}
&T_a [f_{\mu, \sigma^2, \omega} ](y) \notag\\
&= 2 \phi(y)   \rme^{-\frac{\mu^2}{2 \sigma^2}} \sqrt{ \frac{\pi/2}{b} } \exp \left(  \frac{(y+\frac{\mu}{\sigma^2}+jw)^2}{2b} \right) \notag\\
& \qquad \cdot \erf \left(  \frac{ \frac{\mu}{\sigma^2}+jw}{  \sqrt{2 b} }\right)\\
&=2  \frac{1}{\sqrt{2 \pi} }   \sqrt{ \frac{\pi/2}{b} } \exp \left(  \frac{(y+\frac{\mu}{\sigma^2}+jw)^2}{2b} -\frac{\mu^2}{2 \sigma^2}-\frac{y^2}{2} \right) \notag\\
& \quad \cdot  \erf \left(  \frac{ \frac{\mu}{\sigma^2}+jw}{  \sqrt{2 b} }\right)\\
&=   \sqrt{ \frac{1}{b} } \exp \left(  \frac{(y+\frac{\mu}{\sigma^2}+jw)^2}{2b} -\frac{\mu^2}{2 \sigma^2}-\frac{y^2}{2} \right) \notag\\
& \quad \cdot  \erf \left(  \frac{ \frac{\mu}{\sigma^2}+jw}{  \sqrt{2 b} }\right)\\
&=   \sqrt{ \frac{1}{b} } \exp \left(  - \left (1-\frac{1}{b} \right)\frac{(y-\mu)^2}{2} \right) \notag\\
& \quad \cdot\exp \left( \frac{jw (y+\frac{\mu}{\sigma^2})}{b} \right)   \exp \left(- \frac{  \omega^2}{2 b} \right)  \erf \left(  \frac{ \frac{\mu}{\sigma^2}+jw}{  \sqrt{2 b} }\right),
\end{align} 
where we have have used that $b=1+\frac{1}{\sigma^2}$.  This concludes the proof.

\section{Proof of Linearity of the Conditional Expectation} 
\label{app:proofs_of_linearity}

We presented here four different proofs that the only prior distribution that induces linearity of the conditional expectation, under the model in \eqref{eq:Gaussian_noise_model}, is Gaussian $X$. 
To the best of our knowledge, the proofs based on Stein's method and on the cumulant method are new.  

\subsection{Tweedie Formula Approach}
Using Tweedie's formula, the conditional expectation can be written as 
\begin{align}
\bbE[X|Y=y]=y+\frac{\rmd}{\rmd y} \log f_Y(y), \, y \in \bbR
\end{align} 
where $f_Y$ is the pdf of $Y$. Now, by the linearity assumption, we have that 
\begin{align}
(a-1)y= \frac{\rmd}{\rmd y} \log f_Y(y), \,  y \in \bbR . 
\end{align} 
The solution to this differential equation has a unique form and is given by 
\begin{align}
f_Y(y)= \rme^{ (a-1) \frac{y^2}{2}+by+c}
\end{align} 
for some $b$ and $c$.   Therefore, $f_Y$ is Gaussian. Now using the standard characteristic function argument, the only distribution on $X$ that induces $Y$ to be Gaussian is Gaussian.

\subsection{Cumulant Approach}
Our starting place for  this proof is the following expression shown in \cite{metaIdentitiesGaussian}: for $k \ge 1$
\begin{equation}
\frac{\rmd^k}{ \rmd x^k} \bbE[X|Y=y] =\kappa_{X|Y=y}(k+1),
\end{equation} 
where $\kappa_{X|Y=y}(k)$ is the $k$-th order conditional cumulant.  

Now using the linearity assumption, we have that for  $y \in \bbR$
\begin{align}
\kappa_{X|Y=y}(1)&=ay,\\
\kappa_{X|Y=y}(2)&=a,\\
\kappa_{X|Y=y}(k)&=0,  \, k \ge 3
\end{align}
Note that the only distribution that satisfies the above property is Gaussian \cite{lukacs1970characteristic}.
\subsection{Stein Method Approach}

Recall the following two facts. First, by the orthogonality principle, 
\begin{equation}
\bbE[(X-\bbE[X|Y]) g(Y)]=0, \label{eq:orthogonality}
\end{equation} 
for all $g$.  Second,
\begin{equation}
\bbE[f'(U)]= \sigma^2 \bbE[U f(U)] \label{eq:stein_equation}
\end{equation} 
for all differential functions $f$, if and only if $U$ is zero mean Gaussian with variance $\sigma^2$ \cite{stein1972bound}. The result in \eqref{eq:stein_equation} is known as Stein's equation. 

Now using the linearity assumption  and  the orthogonality principle,  observe the following sequence of steps: 
\begin{align}
& \bbE[X g(Y)] \notag\\
 &=  a  \bbE[Y g(Y)]\\
 &= a \bbE[ \bbE[Y g(Y)|X ] ]\\
  &= a \bbE[ \bbE[(X+Z g(X+Z) |X ] ]\\
    &= a \bbE[ X  \bbE[ g(X+Z) |X ] +  \bbE[(Z g(X+Z) |X ] ]\\
    &= a \bbE[ X  \bbE[ g(X+Z) |X ] +  \bbE[ g'(X+Z) |X ] ] \label{Eq:apply_Stein}\\
       &= a \bbE[ X  \bbE[ g(Y) |X ] +  \bbE[ g'(Y) |X ] ]\\
                 &= a \bbE[ X g(Y) +  g'(Y)   ], \label{es:Step_1} 
\end{align} 
where \eqref{Eq:apply_Stein} follows by using Stein's equation. 

Now re-writting \eqref{es:Step_1} we have that, for all $g$
\begin{align}
(1-a)  \bbE[X g(Y)] = a \bbE[ g'(Y)   ] . \label{es:Step_2}
\end{align} 
Next, note that by the orthogonality principle $ \bbE[X g(Y)] =a   \bbE[Y g(Y)]$; therefore \eqref{es:Step_2} can finally be re-written as: for all $g$ 
\begin{equation}
(1-a)  \bbE[Y g(Y)] =  \bbE[ g'(Y)   ] ,
\end{equation} 
which corresponds to Stein's equation and therefore, $Y$ must be Gaussian with variance $\frac{1}{1-a}$. This further implies that $X$ needs to be Gaussian with variance $\frac{a}{1-a}$.

\subsection{Fourier Approach} 
This proof will rely on the orthogonality principle in \eqref{eq:orthogonality}. By the linearity assumption and choose $g(Y)=\rme^{j \omega Y}$, we have that
\begin{align}
\bbE[X \rme^{j \omega Y}] &=a \bbE[Y \rme^{j \omega Y} ]\\
&=- j a \phi_Y'(\omega), \label{eq:Just_after_orthogonality}
\end{align} 
where $\phi_Y(\omega)$ is the characteristic function of $Y$.   Next, note that
\begin{align}
\bbE[X \rme^{j \omega Y}] &= \bbE[X \bbE[ \rme^{j \omega Y}|X] ] \\
&= \bbE[X \rme^{j \omega X} \bbE[ \rme^{j \omega Z}|X] ] \\
&= \bbE[X \rme^{j \omega X}   ] \rme^{-\frac{\omega^2}{2}} \label{eq:Char_function_of_gaussian}\\
&=  -j  \phi_X'(\omega) \rme^{-\frac{\omega^2}{2}}\\
&= -j \left( \frac{\rmd}{\rmd \omega} \phi_X(\omega) \rme^{-\frac{\omega^2}{2}}+ \omega  \phi_X(\omega) \rme^{-\frac{\omega^2}{2}} \right) \\
&= -j \left( \phi_Y'(\omega)+ \omega \phi_Y(\omega) \right) , \label{eq:last_expression_1}
\end{align} 
where in \eqref{eq:Char_function_of_gaussian} we have used the fact that $Z$ is standard normal independent of $X$ and, hence, $\bbE[ \rme^{j \omega Z}|X] =\rme^{-\frac{\omega^2}{2}}$; and in \eqref{eq:last_expression_1} we have used that  $\phi_Y(\omega)= \phi_X(\omega) \rme^{-\frac{\omega^2}{2}}$

Now combining \eqref{eq:Just_after_orthogonality} and \eqref{eq:last_expression_1} we arrive at
\begin{align}
(1-a)\phi_Y'(\omega)+ \omega \phi_Y(\omega) =0, \forall \omega
\end{align} 

Clearly, the only solution to the above differential equation is Gaussian with variance $\frac{1}{1-a}$. This concludes the proof.

\section{Proof of Theorem~\ref{thm:symmetry_of_post} }
\label{app:thm:symmetry_of_post}

The direct part of the theorem follows since the conditions that $X|Y=y \sim \mathcal{N}( \frac{\sigma^2_X}{1+ \sigma_X^2 } y,  \frac{\sigma^2_X}{1+ \sigma_X^2 }) $  and is symmetric around the mean $ \frac{\sigma^2_X}{1+ \sigma_X^2 } y$. 

For the converse part, note that if  $X|Y=y$ for all $y \in S$, then for all $y\in S$, we have that the third conditional moment\footnote{ It can be shown that $X|Y=y$ is always sub-Gaussian \cite{guo2011estimation}. Therefore, all moments exist. }  is given by 
\begin{align}
0&=\bbE[ (X-\bbE[X|Y])^3 | Y=y]\\
&=\kappa_{X|Y=y}(3) ,
\end{align}
where $\kappa_{X|Y=y}(3) $ denotes the third conditional cumulant.  Now, since $y \mapsto \kappa_{X|Y=y}(3) $ is a real-analytic function \cite[Lem.~2]{metaIdentitiesGaussian} and the set $S$ has an accumulation point by the identity theorem \cite{krantz2002primer}, we have that 
\begin{equation}
0=\kappa_{X|Y=y}(3),  \forall y \in S \Longrightarrow  0=\kappa_{X|Y=y}(3),  \forall y \in \bbR.  \label{eq:After_identity_theorem}
\end{equation}

Using the result in \cite[eq.~(55)]{metaIdentitiesGaussian}, the conditional cumulant can be expressed as 
\begin{equation}
\kappa_{X|Y=y}(3) =\frac{\rmd^3}{ \rmd y^3} \log f_Y(y), \forall y \in \mathbb{R} \label{eq:Cumulant_pdf_y_formula}
\end{equation} 
where $f_Y$ is the pdf of $Y$. Therefore, combining \eqref{eq:After_identity_theorem} and \eqref{eq:Cumulant_pdf_y_formula}, we have that   
\begin{equation}
\frac{\rmd^3}{ \rmd y^3} \log f_Y(y)=0, \forall y \in \bbR,
\end{equation} 
which has a unique form of a solution given by 
\begin{equation}
f_Y(y)=\rme^{ay^2+by+c},  y \in \bbR, \label{eq:Sol_of_dfe}
\end{equation} 
for some constants $a,b,c$. 
Now using the standard characteristic function argument, the only distribution on $X$ that induces an output pdf of the form in \eqref{eq:Sol_of_dfe} is Gaussian. This concludes the proof.

\section{Proof of Theorem~\ref{thm_gen_lp_case}}
\label{app:thm_gen_lp_case}

Now consider the case of general $p\in [1,\infty)$. 
Since this is an odd function, the Fourier transform can be calculated by integrating against $i\sin(wx)$ on $(0,\infty)$, so we are led to the definition:
\begin{align}
f_p(w):=\int_0^{\infty}x^{p-1}\rme^{-x^2}\sin(wx) \rmd x.
\end{align}
From the previous analysis, we see that the existence of non-Gaussian prior is equivalent to the existence of nonzero roots of $f_p$.
We first observe that $f_p$ is characterized by an ordinary differential equation:
\begin{lem}
For any $p\in(0,\infty)$, $f_p$ is a smooth function, and
\begin{align}
2f_p''(w)+(p-1)f_p(w)+(wf_p(w))'=0
\label{e109}
\end{align}
for all $w\in\mathbb{R}$.
\end{lem}
\begin{proof}
We have 
\begin{align}
&(wf_p(w))' \notag\\
&=\int_0^{\infty}x^{p-1}\rme^{-x^2}(w\sin(wx))'\rmd x
\\
&=\int_0^{\infty}x^{p-1}\rme^{-x^2}(\sin(wx)+xw\cos(wx))\rmd x,
\end{align}
and using integration by parts,
\begin{align}
\int_0^{\infty}x^p \rme^{-x^2}w\cos(wx) \rmd x
&=-\int_0^{\infty} \hspace{-0.1cm}\rmd(x^p\rme^{-x^2})\sin(wx).
\end{align}
Then \eqref{e109} easily follows.
\end{proof}

\begin{lem}\label{lem16}
For $p\in(2k,2k+2]$, where $k\in\{0,1,2,\dots\}$, $f_p(w)$ has $k$ roots on $(0,\infty)$.
\end{lem}
\begin{proof}
Note that $f_p(w)$ is a smooth, odd function that vanishes at infinity.
Now suppose $p\in(0,2]$, and note that for $m\in\{0,2,4,\dots,\}$ the $m$-th derivative $f^{(m)}_p(w)=f_{p+m}(w)$.
We argue by induction, with the following induction hypotheses for $m=1,2,\dots$:
\begin{itemize}
\item $f_p^{(m)}$ has $m+1$ roots on $(-\infty,\infty)$;
\item at these $m+1$ roots, $f_p^{(m+1)}$ are nonzero and their signs are alternating.
\end{itemize}
The second hypothesis actually follows from the first: Using \eqref{e109} we obtain
\begin{align}
2f_p^{(m+2)}(w)+(m+p)f^{(m)}(w)+wf_p^{(m+1)}(w)=0.
\label{e119}
\end{align}
So if $f_p^{(m)}(w_0)=f_p^{(m+1)}(w_0)$ at some $w_0$, we obtain $f_p^{(m+2)}(w_0)=0$ as well, and the ODE gives the trivial solution $f_p^{(m)}=0$, a contradiction.
Therefore $f_p^{(m+1)}$ must be nonzero at the roots of $f_p^{(m)}$, and by continuity their signs must be alternating (i.e., signs of $f_p^{(m+1)}$ must be different at consecutive roots of $f_p^{(m)}$).

If the induction hypothesis is true for $M$, then the $M+1$ roots partition $\mathbb{R}$ into $M+2$ intervals, so that by Rolle's theorem, there is at least root for $f_p^{(M+1)}$ on the interior of each of these intervals. 
There cannot be more roots: Suppose otherwise, that $a_1$ and $a_2$ are two consecutive roots of $f_p^{(M)}$, and $b_1,b_2\in(a_1,a_2)$ ($b_1<b_2$) are two roots of $f_p^{(M+1)}$.
Then $f_p^{(M)}$ does not change sign on $(a_1,a_2)$, and we can assume without loss of generality that the sign is positive on that interval. 
From \eqref{e119} we see that $f_p^{(M+2)}(b_1),f_p^{(M+2)}(b_2)<0$.
So for sufficiently small $\epsilon>0$ we have $f_p^{(M+1)}(b_1+\epsilon)<0$ and $f_p^{(M+1)}(b_2-\epsilon)>0$, and by continuity there exists $b_3\in(b_1,b_2)$ such that $f_p^{(M+1)}(b_3)=0$ and $f_p^{(M+2)}(b_3)\ge0$.
Then \eqref{e119} does not hold at $b_3$, a contradiction. Similar arguments can be applied when $a_1=-\infty$ or $a_2=\infty$.
Therefore the induction hypotheses are true for all $m\in\{0,1,2,\}$, and in particular the case of even $m$ implies the lemma, noting that $f_{p+m}$ is an odd function vanishing at 0.
\end{proof}

This implies that for $p\in(2,\infty)$, the existence of nonzero roots of $f_p$, which yield the desired non-Gaussian priors.


\begin{lem}
Suppose that $p\in(0,2)$. There exists a constant $c_p\neq 0$ such that 
\begin{align}
f_p(w)=c_p\, {\rm p.v.}\int_{\mathbb{R}}
t^{-p}\rme^{-(w-t)^2/4} \rmd t.
\label{e113}
\end{align}
Note that since $\rme^{-(w-t)^2/4}$ is smooth at $t=0$ for every $w$, the integral is well-defined in the sense of Cauchy's principal value.
As a consequence, whenever $p\in(0,2]$, $f_p(w)=0$ only for $w=0$.
\end{lem}
\begin{proof}
If $p\in(1,2)$, we can prove \eqref{e113} using the fact that the Fourier transform of $|x|^{p-2}$ is $|w|^{1-p}$ for $p\in(1,2)$ \cite[Section~5.9]{lieb2001analysis},
which implies that the Fourier transform of $|x|^{p-1}\sign(x)$ is $|w|^{-p}\sign(w)$ (all up to multiplicative constants).
However, this argument does not directly extend to the case of $p\in(0,1]$,
since 
$|x|^{p-2}$ will then no longer be a tempered distribution (the singularity at 0 leads to divergent integral).

Instead, here we prove \eqref{e113} by showing that $g(w)$, defined as the integral on the right side of \eqref{e113}, must satisfy \eqref{e109}, and so $f_p$ and $g$ are two solutions to the same second-order ODE with $f_p(0)=g(0)$ and $f_p'(0),g'(0)\neq 0$.
Indeed, we have 
\begin{align}
g''(w)={\rm p.v.}\int t^{-p}\left[-\frac1{2}+ \left(\frac{w-t}{2} \right)^2\right]\rme^{-(w-t)^2/4} \rmd t,
\end{align}
and 
\begin{align}
(wg(w))'={\rm p.v.}\int t^{-p}\left[1-\frac{w(w-t)}{2}\right]\rme^{-(w-t)^2/4} \rmd t.
\end{align}
With some arrangements, we see the ODE is equivalent to
\begin{align}
{\rm p.v.}\int t^{-p}\left[a-\frac{t(w-t)}{2}\right]\rme^{-(w-t)^2/4} \rmd t=0.
\end{align}
This is true since using integration by parts, 
\begin{align}
&\int t^{-p+1}\rmd \rme^{-(w-t)^2/4} \notag\\
&=\lim_{\epsilon\downarrow 0}\left(\left.t^{-p+1}\rme^{-(w-t)^2/4}\right|_{\epsilon}^{\infty}
+\left.t^{-p+1} \rme^{-(w-t)^2/4}\right|_{-\infty}^{-\epsilon}\right) \notag\\
&\qquad +(p-1) {\rm p.v.}\int t^{-p}\rme^{-(w-t)^2/4}
\\
&=(p-1) {\rm p.v.}\int t^{-p}\rme^{-(w-t)^2/4}.
\end{align}

Now that \eqref{e113} is verified, it is easy to see from this representation that $w=0$ is the only root of $f_p$, assuming $p\in(0,2)$. On the other hand, $f_2(w)$ can be directly calculated as a linear function times a Gaussian function, so $w=0$ is still the only root.
\end{proof}

\paragraph*{Proof of Theorem~\ref{thm_gen_lp_case}:}
For $p\in[1,2]$, the claim follows since by Lemma~\ref{lem16}, $f_p$ only has a root at 0, and the same proof for the median estimator applies (once reduced to the case of point-supported distributions, we only used $f_p'(0)\neq 0$ to conclude that the distribution is a Dirac delta).

 \bibliography{refs.bib}
 \bibliographystyle{IEEEtran}

\begin{IEEEbiographynophoto}{Leighton P. Barnes} (Member, IEEE) received the Ph.D. in Electrical Engineering from Stanford University in '21. Before that, he received a B.S. in Mathematics ’13, B.S. in Electrical Science and Engineering ’13, and M.Eng. in Electrical Engineering and Computer Science ’15, all from the Massachusetts Institute of Technology. He has received the Harold L. Hazen Award for excellence in teaching at MIT, as well as an IEEE GLOBECOM Best Paper Award in '20. He is currently a permanent member of the research staff at the Center for Communications Research in Princeton, NJ.
\end{IEEEbiographynophoto}

\begin{IEEEbiographynophoto}{Alex Dytso} (Senior Member, IEEE) received the Ph.D. degree from the Department of Electrical and Computer Engineering, University of Illinois, Chicago, in 2016. From September 2016 to August 2020, he was a PostDoctoral Associate with the Department of Electrical Engineering, Princeton University. From 2020 to 2022, he was an Assistant Professor with the Department of Electrical and Computer Engineering, New Jersey Institute of Technology (NJIT). Currently, he is a Staff Engineer with Qualcomm Flarion Technologies Inc. His current research interests include the areas of multi-user information theory and estimation theory, and their applications in wireless networks.
\end{IEEEbiographynophoto}

\begin{IEEEbiographynophoto}{Jingbo Liu} (Member, IEEE)
received the B.S. degree in Electrical Engineering from Tsinghua University, Beijing, China in 2012, and the M.A. and Ph.D. degrees in Electrical Engineering from Princeton University, Princeton, NJ, USA, in 2014 and 2017. He was a Norbert Wiener Postdoctoral Research Fellow at the MIT Institute for Data, Systems, and Society (IDSS) during 2018-2020.
Since 2020, he has been an assistant professor in the Department of Statistics and an affiliate in the Department of Electrical and Computer Engineering at the University of Illinois, Urbana-Champaign, IL, USA.

His research interests include information theory, high dimensional statistics and probability, and machine learning. His undergraduate thesis received the best undergraduate thesis award at Tsinghua University (2012). He gave a semi-plenary presentation at the 2015 IEEE Int. Symposium on Information Theory, Hong-Kong, China. He was a recipient of the Princeton University Wallace Memorial Honorific Fellowship in 2016. His Ph.D. thesis received the Bede Liu Best Dissertation Award of Princeton and the Thomas M. Cover Dissertation Award of the IEEE Information Theory Society (2018).
His coauthored paper was selected as ICML spotlight (3.5\%) in 2024.
\end{IEEEbiographynophoto}

\begin{IEEEbiographynophoto}{H. Vincent Poor} (S'72, M'77, SM'82, F'87) received the Ph.D. degree in EECS from
Princeton University in 1977. From 1977 until 1990, he was on the faculty of the
University of Illinois at Urbana-Champaign. Since 1990 he has been on the faculty at
Princeton, where he is currently the Michael Henry Strater University Professor. During
2006 to 2016, he served as the dean of Princeton’s School of Engineering and Applied
Science, and he has also held visiting appointments at several other universities,
including most recently at Berkeley and Cambridge. His research interests are in the areas
of information theory, machine learning and network science, and their applications in
wireless networks, energy systems and related fields. Among his publications in these
areas is the book {\it Machine Learning and Wireless Communications}. (Cambridge
University Press, 2022). Dr. Poor is a member of the National Academy of Engineering
and the National Academy of Sciences and is a foreign member of the Royal Society and
other national and international academies. He received the IEEE Alexander Graham Bell
Medal in 2017.
\end{IEEEbiographynophoto}

\end{document}